\documentclass[12pt]{amsart}
\usepackage{mabliautoref}
\usepackage{amsmath}
\usepackage{amsfonts}
\usepackage{amssymb}
\usepackage{amscd}
\usepackage{graphicx}
\RequirePackage[dvipsnames,usenames]{color}
\usepackage{soul,xcolor}
\setstcolor{red}
\usepackage{stmaryrd}
\SetSymbolFont{stmry}{bold}{U}{stmry}{m}{n}
\usepackage{mathtools}
\usepackage{booktabs}
\usepackage{multirow}\usepackage{subcaption}

\usepackage{mathdots}
\newtagform{tiny}{\tiny(}{)}

\usepackage{mathtools}
\usepackage{hyperref}
\usepackage[margin=1.25in]{geometry}

\usepackage{amsthm}
\usepackage{comment}
\usepackage[all,cmtip]{xy}
\usepackage{tikz-cd}
\usetikzlibrary{cd}

\usepackage[all]{xy}

\usepackage{etoolbox}
\apptocmd{\sloppy}{\hbadness 10000\relax}{}{}

\usepackage{enumitem}

\title{On Grauert--Riemenschneider vanishing for Cohen--Macaulay schemes of klt type}

\author[J. Baudin]{Jefferson Baudin}
  \address{\'Ecole Polytechnique F\'ed\'erale de Lausanne, Chair of Algebraic Geometry \newline 
    \indent MA C3 575 (Bâtiment MA), Station 8, CH-1015 Lausanne}
  \email{jefferson.baudin@epfl.ch}

\author[T. Kawakami]{Tatsuro Kawakami}
\address{Graduate School of Mathematical Sciences, University of Tokyo, 3-8-1 Komaba, \newline
\indent Meguro-ku, Tokyo 153-8914, Japan}
\email{tatsurokawakami0@gmail.com}

 \author[L. Rösler]{Linus Rösler}
  \address{\'Ecole Polytechnique F\'ed\'erale de Lausanne, Chair of Algebraic Geometry \newline 
    \indent MA C3 615 (Bâtiment MA), Station 8, CH-1015 Lausanne}
  \email{linus.rosler@epfl.ch}

\hypersetup{
bookmarks,
bookmarksdepth=3,
bookmarksopen,
bookmarksnumbered,
pdfstartview=FitH,
colorlinks,backref,hyperindex,
linkcolor=Sepia,
anchorcolor=BurntOrange,
citecolor=MidnightBlue,
citecolor=OliveGreen,
filecolor=BlueViolet,
menucolor=Yellow,
urlcolor=OliveGreen
}

\def\phi{\varphi}
\def\epsilon{\varepsilon}

\def\Spec{\operatorname{Spec}}

\def\Supp{\operatorname{Supp}}
\def\codim{\operatorname{codim}}
\def\Pic{\operatorname{Pic}}

\def\Exc{\operatorname{Exc}}

\def\max{\operatorname{max}}

\def\m{{\mathfrak m}}

\newcommand{\Q}{\mathbb{Q}} 
 
\newcommand{\bD}{\mathbb{D}} 
 
\newcommand{\Z}{\mathbb{Z}}

\newcommand{\sO}{\mathcal{O}}
\newcommand{\cO}{\mathcal{O}}
\newcommand{\sHom}{\mathcal{H}\! \mathit{om}}

\newcommand{\sH}{\mathcal{H}}
\newcommand{\cH}{\mathcal{H}}

\newcommand{\JB}[1]{\textcolor{blue}{(JB: #1)}}

\newsavebox{\pullback}
\sbox\pullback{%
\begin{tikzpicture}%
\draw (0,0) -- (1ex,0ex);%
\draw (1ex,0ex) -- (1ex,1ex);%
\end{tikzpicture}}

\newsavebox{\pullbackdl}
\sbox\pullbackdl{%
\begin{tikzpicture}%
\draw (-1ex,0ex) -- (0ex,0ex);%
\draw (0ex,-1ex) -- (0ex,0ex);%
\end{tikzpicture}}

\newsavebox{\pushoutdr}
\sbox\pushoutdr{%
\begin{tikzpicture}%
\draw (-1ex,-1ex) -- (-1ex,0ex);%
\draw (-1ex,0ex) -- (0ex,0ex);%
\end{tikzpicture}}

\newcommand{\stacksproj}[1]{{\cite[Tag~\href{http://stacks.math.columbia.edu/tag/#1}{#1}]{stacks-project}}}

\newcommand{\stacksprojs}[2]{\cite[Tags ~\href{http://stacks.math.columbia.edu/tag/#1}{#1} and ~\href{http://stacks.math.columbia.edu/tag/#2}{#2}]{stacks-project}}

\numberwithin{equation}{thm}

\begin{document}
\tolerance = 9999

\begin{abstract}
Given a Cohen--Macaulay scheme of klt type $X$ and a resolution $\pi\colon Y\to X$, we show that $R^1\pi_{*}\omega_Y=0$. We deduce that if $\dim(X) = 3$, then $X$ satisfies Grauert--Riemenschneider vanishing and therefore has rational singularities. We also obtain that in arbitrary dimension, if $X$ is of finite type over a perfect field of characteristic $p > 0$, then $X$ has $\Q_p$--rational singularities.
\end{abstract}

\maketitle
\markboth{}{}

\section{Introduction}

In this paper, we study \emph{Grauert--Riemenschneider vanishing}, a relative version of Kodaira vanishing, for
Noetherian excellent integral schemes.

\begin{defn}
Let $X$ be a Noetherian excellent integral scheme of finite dimension with a dualizing complex. We say that $X$ satisfies \emph{Grauert--Riemenschneider vanishing} if for every resolution $\pi \colon Y \to X$, we have $R^j \pi_* \omega_Y = 0$ for all $j \geq 1$.
\end{defn}

Grauert--Riemenschneider vanishing is a fundamental result in characteristic zero birational geometry. It is for example used to show that klt singularities are rational \cite{Elkik_Rationalite_des_singularites_canoniques, Kovacs_Characterisation_of_rational_singularities}, or that rational singularities are stable under deformations \cite{Elkik_singularites_rationnelles_et_deformations}. 

However, this vanishing is known to fail in every positive characteristic. Such examples can be constructed by taking affine cones over smooth projective surfaces that fail to satisfy Kodaira vanishing \cite{Raynaud_Contre-exemple_au_vanishing_theorem_en_caracteristique_p>0,Hacon_Kovacs_Generic_vanishing_fails_for_singular_varieties_and_in_characteristic_p>0}, or by taking wild $\mathbb{Z}/p\mathbb{Z}$--quotients \cite{Totaro_The_failure_of_KV_and_terminal_singularities_that_are_not_CM,Totaro_Terminal_3folds_that_are_not_CM,BBK}.

On the other hand, Grauert--Riemenschneider vanishing is sometimes known to hold when the singularities are mild. For instance, it has been shown that three--dimensional klt singularities in characteristic  $p > 5$ satisfy the vanishing theorem \cite{Hacon-Witaszek,Bernasconi-Kollar}. Moreover, this bound on the characteristic is optimal, as counterexamples are known in characteristics $2$, $3$, and $5$ \cite{Ber,CT19-2,ABL}. Notably, all known counterexamples fail to be Cohen--Macaulay, which naturally raises the question: Do \emph{Cohen--Macaulay} klt singularities satisfy Grauert--Riemenschneider vanishing?

This question is particularly relevant for applications, as \emph{strongly $F$--regular} singularities---regarded as characteristic $p$ analogues of klt singularities \cite{Takagi,SS10}---are known to be Cohen--Macaulay.

In this paper, we prove that Cohen--Macaulay klt singularities satisfy Grauert--Riemenschneider vanishing in degree one. 

\begin{thm}\label{introthm:general dim}
   Let $X$ be a Noetherian excellent normal scheme of klt type, and let $\pi\colon Y\to X$ be a resolution. Then the following hold:
   \begin{enumerate}[label = \textup{(\arabic*)}]
       \item\label{itm:main_thm:first} If $X$ is of klt type, then $R^{d-1}\pi_{*}\sO_Y=0$.
       \item\label{itm:main_thm:second} If $X$ is of klt type and Cohen--Macaulay, then $R^{1}\pi_{*}\omega_Y=0$.
   \end{enumerate} 
    \end{thm}
\begin{rem}
\begin{enumerate}
    \item It is well--known that $R^{d-1}\pi_{*}\omega_Y=0$ always holds (see \autoref{prop:GR in general dim}).
    \item By \cite[Theorem 1.3]{Ishii-Yoshida}, we also have $R^{1}\pi_{*}\sO_Y=0$ in the situation of \autoref{introthm:general dim}.\ref{itm:main_thm:second} (see \cite[Theorem 1.1]{Ishii-Yoshida} for a more general statement).
    \item We cannot drop the assumption of Cohen--Macaulayness in \autoref{introthm:general dim}.\ref{itm:main_thm:second} \cite{Ber,CT19-2,ABL}.
\end{enumerate}
\end{rem}

As an immediate corollary, we obtain the following:

\begin{COR}\label{introcor:three dim}
    Let $X$ be a Noetherian excellent normal scheme of dimension 3. If $X$ is of klt type and Cohen--Macaulay, then $X$ satisfies Grauert--Riemenschneider vanishing and has rational singularities.
\end{COR}

In particular, a strongly $F$--regular (or even quasi--$F$--regular, see \cite{Tanaka_Witaszek_Yobouko_quasi_Fe_splittings_and_quasi_F_regularity}) threefold in positive characteristic satisfies Grauert--Riemenschneider vanishing and has rational singularities (see \autoref{thm:GR for SFR}). Similarly, we also obtain such a result in the globally $+$--regular setting (see \autoref{thm:vanishing_+_regular}).\medskip

As we already pointed out, klt singularities need not be rational in positive characteristic. Nevertheless, it is expected that they satisfy a weak notion: Witt--rationality \cite{Rulling_Chatzistamatiou_Hodge_Witt_and_Witt_rational,Blickle_Esnault_Rational_singularities_and_rational_points}. Briefly, a normal variety $X$ admitting a resolution $\pi \colon Y \to X$ is said to have Witt--rational singularities if, for all $i> 0$, the sheaves $R^i\pi_*W\cO_Y$ are annihilated by some fixed $p$--power.

The fact that klt singularities are Witt--rational is known to be true in dimension 3 \cite{GNT16, Hacon_Witaszek_On_the_relative_MMP_for_threefolds_in_low_char}, and in dimension 4 if one assumes the existence of log resolutions for all birational models and that $p > 5$ \cite{Hacon_Witaszek_On_the_relative_MMP_for_4folds_in_pos_and_mixed_char}. However, this question is widely open in general. Here, we present a version of this statement that holds in any dimension:

\begin{thm}\label{intro:cor_Qp_rationality}
    Let $X$ be a Cohen--Macaulay integral scheme of klt type which is of finite type over a perfect field of characteristic $p > 0$. Assume that $X$ admits a resolution of singularities. Then $X$ has $\Q_p$--rational singularities.

    If in addition $X$ is projective and has isolated singularities, then it has Witt--rational singularities. 
\end{thm}

\begin{rem}\label{rem_below}
\begin{itemize}
    \item The notion of $\Q_p$--rationality was defined in \cite{PZ} and is a mild weakening of Witt--rationality. It seems that in practice, knowing $\Q_p$--vanishing instead of the full Witt--vanishing is enough for many purposes. Nevertheless, we hope to eventually be able to strengthen \cite[Theorem A]{Baudin_Witt_GR_vanishing_and_applications} to obtain Witt--rationality above without assuming isolated singularities. 
    \item  As the proof shows, one can significantly weaken the Cohen--Macaulay assumption in \autoref{intro:cor_Qp_rationality}. Namely, it is enough to assume $\mathbb{Q}_p$--Cohen--Macaulayness for the first statement and Witt--Cohen--Macaulayness for the second one (see \cite[Definition 5.1.3]{Baudin_Witt_GR_vanishing_and_applications}). For example, these notions are invariant under universal homeomorphisms and arbitrary finite quotients, unlike the usual Cohen--Macaulayness \cite{Forgarty_On_the_depth_of_local_rings_of_invariants_of_cyclic_groups}. We hope to be eventually able to show that only the klt type assumption and the existence of one resolution is enough. 
\end{itemize}
\end{rem}





\section{Preliminaries}

\subsection{Notation and terminology}
Throughout, a \emph{variety} denotes an integral, excellent, Noetherian scheme that admits a dualizing complex. A \emph{pair} $(X, \Delta)$ consists of a normal variety $X$ together with an effective $\Q$--divisor $\Delta$ on $X$.

All dualizing complexes are normalized in the sense of \cite{Hartshorne_Residues_and_duality}. That is, if $X$ is a variety of dimension $d$ with a dualizing complex $\omega_X^{\bullet}$, then $\cH^{i}(\omega_X^{\bullet}) = 0$ for all $i < -d$, and $\omega_X \coloneqq \cH^{-d}(\omega_X^{\bullet}) \neq 0$ (given a complex $\mathcal{A}^{\bullet}$ in some derived category and $i \in \Z$, we let $\cH^i(\mathcal{A}^{\bullet})$ denotes its $i$--th cohomology object). 

If we fix a variety $X$ with a dualizing complex $\omega_X^{\bullet}$ as above and $\pi \colon Y \to X$ is a separated morphism of finite type, then we naturally induce a dualizing complex on $Y$ by taking $\omega_Y^{\bullet} \coloneqq \pi^!\omega_X^{\bullet}$ (see \stacksproj{0AA3}).

A resolution of a variety $X$ is a projective birational morphism $\pi \colon Y \to X$ with $Y$ regular.

\begin{defn}
    We say a variety $X$ has \emph{rational singularities} if it is Cohen--Macaulay, and for any resolution $\pi \colon Y \to X$, the natural map $\cO_X \to R\pi_*\cO_Y$ is an isomorphism. 
\end{defn}
Note that by Grothendieck duality, a variety with rational singularities automatically satisfies Grauert--Riemenschneider vanishing. Thanks to \cite{CR15}, one only needs to check rationality for one resolution, if one assumes resolutions of singularities. In positive characteristic, it is not needed to assume the existence of resolutions by \cite{Chatzistamatiou_Rulling_Higher_direct_images_of_the_structure_sheaf_in_positive_characteristic}.

\begin{defn}
    We say that a variety $X$ is \textit{of klt type} if it is normal and there exists an effective $\Q$--divisor $\Delta$ such that the pair $(X,\Delta)$ is klt (see \cite[Definition 2.28]{BLP+}).
\end{defn}

\section{Proofs of the main theorems}

\subsection{Rationality results}

In this section, we prove \autoref{introthm:general dim} and \autoref{introcor:three dim}.

\begin{prop}\label{prop:GR in general dim}
Let $X$ be a variety of dimension $d$, and let $\pi \colon Y \to X$ be a resolution. Then $R^{d-1}\pi_{*}\omega_Y=0$ for every resolution $\pi\colon Y\to X$.
\end{prop}
\begin{proof}
 We may assume that $X=\Spec R$ is affine, and let $\pi \colon Y \to X$ be a resolution. Since $Y$ is projective over $R$, we can find a general hyperplane section $H$ such that $H$ is smooth by \cite[Theorem 2.17]{BLP+} and $R^{d-1}\pi_{*}\omega_X(H)=0$ by relative Serre vanishing.
 Consider the short exact sequence
 \[
 0\to \sO_X(-H) \to \sO_X \to\sO_H \to 0.
 \]
 Taking $\sHom_{\sO_X}(-,\omega_X)$, we have a short exact sequence
 \[
 0\to \omega_X \to \omega_X(H) \to \omega_H \to 0.
 \]
 Since $R^{d-1}\pi_{*}\omega_X(H)=0$, the desired vanishing $R^{d-1}\pi_{*}\omega_X=0$ can be reduced to the vanishing
 $R^{d-2}\pi_{*}\omega_H=0$.
 By repeating this argument, we can reduced to the case $\dim X=2$, which follows from \cite[Theorem 10.4]{Kol13}.
\end{proof}

\begin{thm}\label{thm:weak rational in general dim}
Let $X$ be a variety of klt type, and let $\pi\colon Y\to X$ be a resolution. Then for all $i > 0$, \[\codim R^i\pi_*\sO_Y > i + 1.\] In particular, $R^{d-1}\pi_{*}\sO_Y=0$ where $d = \dim(X)$.
\end{thm}
\begin{rem}
    \begin{itemize}
        \item  Note that we only require $\pi\colon Y\to X$ to be a resolution and not a log resolution. In fact, we only need that $Y$ is normal and factorial. By an example of Linquan Ma (see \cite[Example 3.2]{Ishii-Yoshida}), one can probably not weaken these hypotheses much.
        \item In particular, our exceptional divisors might not be normal a priori (only integral) since we do not assume log resolutions. Although we will take a divisorial notation below which may give the idea that normality is needed, we will really work with $\Q$--line bundles (i.e. elements of $\Pic \otimes_{\Z} \Q$) on the resolution and on the exceptional components.
    \end{itemize}
\end{rem}

\begin{proof}
    By induction on the dimension and by localizing, it is enough to show that $R^{d - 1}\pi_*\cO_Y = 0$. By \cite[Theorem II.7.17]{Har}, there exists a closed subscheme $Z\subseteq X$ such that $\pi$ is the blow--up of $X$ along $Z$. Let us denote $E=\pi^{-1}(Z)$, so that $\sO_{Y}(-E)$ is $\pi$--ample by \stacksprojs{02NS}{02OS}. We can take $n\gg 0$ such that $R^{d-1}\pi_*\sO_Y(-nE)=0$ by relative Serre vanishing. Let us write
    \begin{align*}
        nE=\underbrace{\sum_{i\in I}r_iF_i}_{F}\:+\:G,
    \end{align*}
    for some positive integers $r_i>0$,
    where the $F_i$'s are exactly the $\pi$--exceptional (i.e.~the codimension of the image is at least $2$) components of $E$. Also, observe that $R^{d-1}\pi_*\sO_Y(-F)=0$. Indeed, we have the short exact sequence
    \begin{align*}
        0\to\sO_Y(-nE)\to\sO_Y(-F)\to\sO_G(-F)\to 0.
    \end{align*}
    As $R^{d-1}\pi_*\sO_Y(-nE)=0$, it suffices to show that $R^{d-1}\pi_*\sO_G(-F)=0$. Given that all fibers of $G \to \pi(G)$ have dimension $\leq d - 2$, this is immediate. To conclude the proof, we are then left to show the following:
    \begin{cl}
        If $R^{d-1}\pi_{*}\sO_Y(-\sum_{i\in I} n_iF_i)=0$ for some $(n_i)_{i\in I}\in\Z_{\geq 0}$ satisfying $\sum_{i\in I} n_i\geq 1$, then there exists $j\in I$ such that $n_j\geq 1$ and \[R^{d-1}\pi_{*}\sO_Y\left(-(n_j-1)F_j-\sum_{i\in I\setminus\{j\}} n_iF_i\right)=0.\] 
    \end{cl}
    \noindent \textit{Proof of the claim.}
    For now, fix $j \in I$ with $n_j \geq 1$. We will pick a specific $j$ later. Consider the short exact sequence
    \begin{multline*}
       0\to \sO_Y\left(-\sum_{i\in I} n_iF_i\right)\to  \sO_Y\left(-(n_j-1)F_j-\sum_{i\in I\setminus\{j\}} n_iF_i\right)\\\to \sO_{F_j}\left(F_j-\sum_{i\in I} n_iF_i\right)\to 0. 
    \end{multline*}
     Since, we aim to show that \[R^{d-1}\pi_{*}\sO_Y\left(-(n_j-1)F_j-\sum_{i\in I\setminus\{j\}} n_iF_i\right)=0,\] this is equivalent to proving that \[R^{d-1}\pi_{*}\sO_{F_j}\left(F_j-\sum_{i\in I} n_iF_i\right)=0.\]
    If $\dim \pi(F_j)>0$, then this is immediate since then fibers of $F_j \to \pi(F_j)$ have dimension $\leq d - 2$. If $\dim \pi(F_j)=0$, then
    \begin{align*}
       R^{d-1}\pi_{*}\sO_{F_j}\left(F_j-\sum_{i\in I} n_iF_i\right)&=H^{d-1}\left(F_j, \sO_{F_j}\left(F_j-\sum_{i\in I} n_iF_i\right)\right)\\
    &\cong H^0\left(F_j, \sO_{F_j}\left(K_{F_j}-F_j+\sum_{i\in I} n_iF_i\right)\right)^{\vee} \\
    &\cong H^0\left(F_j, \sO_{F_j}\left(K_{Y}+\sum_{i\in I} n_iF_i\right)\right)^{\vee}.
    \end{align*}
    To conclude that the latter group vanishes, it is then enough to show that $-(K_{Y}+\sum_{i\in I} n_iF_i)|_{F_j}$ is $\pi|_{F_j}$--big.
    
    Let us now find some $j \in I$ that gives this. Since $X$ is of klt type, there exists an effective $\Q$--divisor $\Delta$ such that
    \[
    K_Y+\sum_{i\in I}a_iF_i+\pi_{*}^{-1}\Delta\sim_{\Q}\pi^{*}(K_X+\Delta)
    \]
    for some $a_i\in\Q_{<1}$ (note that $\Supp(F)=\Exc(\pi)$, since $\pi$ is an isomorphism outside of $E$). Let $J\coloneqq \{i\in I \mid n_i-a_i> 0\}\subset I$.
    Note that $J\neq \emptyset$ since $\sum_{i\in I} n_i\geq 1$ and $a_i<1$. Let
    \begin{align*}
        t\coloneqq\underset{i\in J}{\max}\left\{\frac{n_i-a_i}{r_i}\right\}\in\Q_{>0},
    \end{align*}
    and let $j\in J$ be an index where the maximum is attained. We then have
    \[F_j\not\subset \Supp\left(t\left(\sum_{i\in I} r_iF_i\right)-\sum_{i\in J} (n_i-a_i )F_i\right),\]
    so 
    \[
    -\left.\left(\sum_{i\in J} (n_i-a_i )F_i)\right)\right|_{F_j}=\left.\left(-tG-t\sum_{i\in I} r_iF_i\right)\right|_{F_j}+\left.\left(tG+t\sum_{i\in I} r_iF_i-\sum_{i\in J} (n_i-a_i )F_i\right)\right|_{F_j}
    \]
    is $\pi|_{F_j}$--big (recall that $-G-\sum_{i\in I} r_iF_i$ is $\pi$--ample), whence
    \begin{align*}
        -\left.\left(K_{Y}+\sum_{i\in I} n_iF_i\right)\right|_{F_j} & \sim_{\Q,\pi|_{F_j}} -\left.\left(\sum_{i\in I} (n_i-a_i)F_i\right)\right|_{F_j}+\pi^{-1}_{*}\Delta|_{F_j}\\
        &= -\left.\left(\sum_{i\in J} (n_i-a_i )F_i\right)\right|_{F_j}+\left.\left(\sum_{i\in I\setminus J} (a_i-n_i )F_i\right)\right|_{F_j}+\pi^{-1}_{*}\Delta|_{F_j}
    \end{align*}
    is also $\pi|_{F_j}$--big. 
\end{proof}

\begin{lem}\label{lem:dual}
    Let $X$ be a normal variety of dimension $d$, and let $\pi\colon Y\to X$ be a resolution. Suppose that 
    \[ \codim R^i\pi_*\sO_Y > i + 1 \] for all $i \geq 1$. Then $\pi_*\omega_Y = \omega_X$ and there is a natural injection \[ R^1\pi_*\omega_Y \hookrightarrow \sH^{-(d - 1)}(\omega_X^{\bullet}). \] In particular, if $X$ is Cohen--Macaulay, then $R^1\pi_*\omega_Y = 0$.
\end{lem}
\begin{proof}
    Consider the exact triangle \[ \begin{tikzcd}
    \cO_X \arrow[rr] &  & R\pi_*\cO_Y \arrow[rr] &  & \tau_{\geq 1}R\pi_*\cO_Y \arrow[rr, "+1"] &  & {}
    \end{tikzcd} \] Applying $\bD(-)\coloneqq \mathcal{RH}om(-, \omega_X^{\bullet})$ and Grothendieck duality gives \[ \begin{tikzcd}
    \bD(\tau_{\geq 1}R\pi_*\sO_Y) \arrow[rr] &  & {R\pi_*\omega_Y[d]} \arrow[rr] &  & \omega_X^{\bullet} \arrow[rr, "+1"] &  & {}
    \end{tikzcd} \] so taking cohomology sheaves induces an exact sequence \[ \begin{tikzcd}
        \pi_*\omega_X \arrow[r] & \omega_Y \arrow[r] & \sH^{- (d - 1)}\bD(\tau_{\geq 1}R\pi_*\sO_Y) \arrow[r] & R^1\pi_*\omega_Y \arrow[r] & {\sH^{- (d - 1)}(\omega_X^{\bullet}).}
    \end{tikzcd} \] It is enough to show that \[ \sH^{- (d - 1)}\bD(\tau_{\geq 1}R\pi_*\sO_Y) = 0. \]
    We will show by descending induction on $i \geq 1$ that $\sH^{- (d - 1)}\bD(\tau_{\geq i}R\pi_*\sO_Y) = 0$. For $i \gg 0$, there is nothing to show. Fix $i \geq 1$, and consider the exact triangle 
    \[  \begin{tikzcd}
    { R^if_*\cO_Y[-i]} \arrow[rr] &  & \tau_{\geq i}Rf_*\cO_Y \arrow[rr] &  & \tau_{\geq i + 1}Rf_*\cO_Y \arrow[rr, "+1"] &  & {}
    \end{tikzcd} \] (see \stacksproj{08J5}). Applying $\bD$ gives \[ \begin{tikzcd}
    \bD(\tau_{\geq i + 1}Rf_*\cO_Y) \arrow[rr] &  & \bD(\tau_{\geq i}Rf_*\cO_Y) \arrow[rr] &  & {\bD(R^if_*\cO_Y)[i]} \arrow[rr, "+1"] &  & {}
    \end{tikzcd} \] Since $\cH^{-(d - 1)}\bD(\tau_{\geq i + 1}Rf_*\cO_Y) = 0$ by the induction hypothesis, it is enough to show that 
    $\cH^{-(d - 1)}(\bD(R^if_*\cO_Y)[i]) = 0$ by the long exact sequence in cohomology sheaves. Given that $\dim(\Supp(R^if_*\cO_Y)) \leq d - i - 2$ by assumption, we know by \stacksproj{0A7U} that $\bD(R^if_*\cO_Y)$ is supported in degrees $\geq -(d -i - 2)$. Equivalently, $\bD(R^if_*\cO_Y)[i]$ is supported in degrees $\geq -(d - 2)$, so $\cH^{-(d - 1)}\bD(R^if_*\cO_Y)[i] = 0$. \qedhere
\end{proof}

\begin{proof}[Proof of \autoref{introthm:general dim}]
    The assertions follow from \autoref{prop:GR in general dim}, \autoref{thm:weak rational in general dim} and \autoref{lem:dual}.
\end{proof}

\begin{proof}[Proof of \autoref{introcor:three dim}]
    Let $\pi \colon Y \to X$ be a resolution. Given that $R^i\pi_*\omega_X = 0$ for all $i > 0$ by \autoref{introthm:general dim} and \autoref{prop:GR in general dim} and that $\pi_*\omega_Y = \omega_X$ by \autoref{lem:dual}, we have that $R\pi_*\omega_Y = \omega_X$. We then deduce that $R\pi_*\cO_Y = \cO_X$ by Grothendieck duality and Cohen--Macaulayness of $X$. 
\end{proof}

For the definition of strongly $F$--regular (resp. quasi--$F$--regular, $+$--regular) singularities, we refer the reader to \cite[Definition 3.1]{SS10} (resp.~\cite[Definition 4.1]{Tanaka_Witaszek_Yobouko_quasi_Fe_splittings_and_quasi_F_regularity}, \cite[Definition 6.21]{BLP+}). Note that by definition, a strongly $F$--regular or quasi $F$--regular variety is $F$--finite (i.e. the absolute Frobenius is finite). We say a pair $(X,\Delta)$ is \textit{$+$-regular} if it is $+$-regular at each stalk.

\begin{thm}\label{thm:GR for SFR}
    Let $X$ be a 3--dimensional strongly $F$--regular variety. Then $X$ satisfies Grauert--Riemenschneider vanishing and
    has rational singularities.
\end{thm}
\begin{proof}
    We know by \cite[Corollary 2.5]{Hochster-Huneke89} and \cite[Remark 13.6]{Gab04} that $X$ is Cohen--Macaulay. Combining \cite[Corollary 6.9]{SS10} and \cite[Theorem 3.3]{Hara-Watanabe}, we deduce that $X$ is of klt type, so the proof is complete by \autoref{introcor:three dim}.
\end{proof}

\begin{thm}\label{thm:vanishing_+_regular}
Let $(X, \Delta)$ be a 3-dimensional $+$--regular pair such that $K_X + \Delta$ is $\Q$--Cartier. Then $X$ satisfies Grauert--Riemenschneider vanishing and has rational singularities.

The same statement holds if $X$ quasi--$F$--regular, and $\Delta = 0$.


\end{thm}
\begin{proof}
    By \cite[Proposition 6.10]{BLP+} (resp. \cite[Theorems 5.8 and 8.9]{KTTWYY3}), we know that the pair $(X, \Delta)$ is klt and that $X$ is Cohen--Macaulay. We then obtain the result by \autoref{introcor:three dim}.
\end{proof}

\subsection{$\mathbb{Q}_p$--rationality}
Throughout, fix a variety $X$ of finite type over a perfect field of positive characteristic. For $n \geq 1$, we let $W_n\cO_X$ denote the sheaf of $p$--typical Witt vectors, with its induced Verschiebung, restrictions and Frobenius maps (see e.g. \cite[Section 2.2]{KTTWYY1}). The complex $W_n\omega_X^{\bullet}$ denotes the canonical dualizing complex on the scheme $W_nX$ given by the locally ringed space $(X, W_n\cO_X)$, and $W_n\omega_X$ denotes its smallest non--zero cohomology sheaf (see \cite[Section 2.2]{Baudin_Witt_GR_vanishing_and_applications}).

\begin{rem}
    Recall that as a set, $W_n\cO_X$ simply consists of $n$--uples in $\cO_X$. The Verschiebung $V \colon F_*W_n\cO_X \to W_{n + 1}\cO_X$ sends $(s_1, \dots, s_n)$ to $(0, s_1, \dots, s_n)$, while the restriction $R \colon W_{n + 1}\cO_X \to W_n\cO_X$ sends $(s_1, \dots, s_{n + 1})$ to $(s_1, \dots, s_n)$. In particular, there is a natural short exact sequence \[ \begin{tikzcd}
        0 \arrow[rr] &  & F_*W_{n}\cO_X \arrow[rr, "V"] &  & W_{n+1}\cO_X \arrow[rr, "R^n"] &  & \cO_X \arrow[rr] &  & 0
    \end{tikzcd} \]
    of $W_{n+1}\sO_X$-modules.
\end{rem}

\begin{lem}\label{klt_gives_weak_Witt_psrat}
    Let $X$ be a klt type variety of finite type over a perfect field, and let $\pi \colon Y \to X$ denote a resolution. Then for all $n \geq 1$, we have that $\pi_*W_n\omega_Y = W_n\omega_X$.
\end{lem}
\begin{proof}
    Let us prove the result by induction on $n \geq 1$. The case $n = 1$ is contained in \autoref{lem:dual}. In general, applying Grothendieck duality (i.e. the functor $\mathcal{RH}om_{W_{n+1}\cO_X}(-, W_{n+1}\omega_X^{\bullet})$) to the diagram
    \[ \begin{tikzcd}
        F_*W_n\cO_X \arrow[rr, "V"] \arrow[d] &  & W_{n + 1}\cO_X \arrow[rr, "R^n"] \arrow[d] &  & \cO_X \arrow[rr, "+1"] \arrow[d] &  & {} \\
        F_*R\pi_*W_n\cO_Y \arrow[rr, "V"]     &  & R\pi_*W_{n+1}\cO_Y \arrow[rr, "R^n"]     &  & R\pi_*\cO_Y \arrow[rr, "+1"]     &  & {}
    \end{tikzcd} \] and taking the long exact sequence in cohomology gives an exact diagram with exact rows
    \[ \begin{tikzcd}
        0 \arrow[r] & \pi_*\omega_Y \arrow[d, "\cong"'] \arrow[r] & \pi_*W_{n+1}\omega_Y \arrow[d] \arrow[r] & F_*W_n\omega_Y \arrow[d, "\cong"] \arrow[r] & R^1\pi_*\omega_Y \arrow[d, hook] \arrow[r]            & \dots \\
        0 \arrow[r] & \omega_X \arrow[r]                          & W_{n + 1}\omega_X \arrow[r]                & F_*W_n\omega_X \arrow[r]                    & \mathcal{H}^{- (d - 1)}(\omega_X^{\bullet}) \arrow[r] & \dots
    \end{tikzcd} \] (note that $R^1\pi_*\omega_Y \to \cH^{-(d - 1)}(\omega_X^{\bullet})$ is injective by \autoref{lem:dual} and \autoref{introthm:general dim}). A diagram--chasing argument then concludes that $\pi_*W_{n +1}\omega_Y \to W_{n+1}\omega_X$ is an isomorphism.
\end{proof}

\begin{proof}[Proof of \autoref{intro:cor_Qp_rationality}]
    This statement and the statement in \autoref{rem_below} follows immediately from \autoref{klt_gives_weak_Witt_psrat} and \cite[Theorem 5.1.4]{Baudin_Witt_GR_vanishing_and_applications} (see also Theorem A in \emph{loc. cit.} to obtain the statements when we assume projectivity and isolated singularities).
\end{proof}

\section*{Acknowledgements}
We would like to thank Fabio Bernasconi and Shou Yoshikawa for useful conversations related to the content of this article.
TK was supported by the JSPS KAKENHI grant number JP24K16897. JB and LR were supported by the ERC starting grant \#804334. 

\bibliography{hoge.bib}
\bibliographystyle{alpha}

\bigskip

\end{document}